\documentclass[12pt]{amsart}
\usepackage{amscd,amsmath,amsthm,amssymb}
\usepackage{amsfonts,amssymb,amscd,amsmath,enumerate,verbatim}
\usepackage{tikz, float} \usetikzlibrary {positioning}
\usepackage[left]{lineno}
\usepackage{pstcol,pst-plot,pst-3d}
\usepackage{color}
\usepackage{pstricks}
\usepackage{stmaryrd}
\usepackage[utf8]{inputenc}
\usepackage{pstricks-add}
\usepackage{graphicx}
\newpsstyle{fatline}{linewidth=1.5pt}
\newpsstyle{fyp}{fillstyle=solid,fillcolor=verylight}
\definecolor{verylight}{gray}{0.97}
\definecolor{light}{gray}{0.9}
\definecolor{medium}{gray}{0.85}
\definecolor{dark}{gray}{0.6}

 \def\NZQ{\mathbb}               
 \def\NN{{\NZQ N}}

 \def\Pc{{\mathcal P}}
 \def\Jc{{\mathcal J}}

 \def\G{{\mathcal G}}

  \def\Gc{{\mathcal G}}
  \def\Ac{{\mathcal A}}

 \def\ab{{\mathbf a}}
 \def\bb{{\mathbf b}}
 \def\xb{{\mathbf x}}

 \def\cb{{\mathbf c}}
 \def\db{{\mathbf d}}
 \def\eb{{\mathbf e}}

 \def\opn#1#2{\def#1{\operatorname{#2}}} 
 \opn\chara{char} \opn\length{\ell} \opn\pd{pd} \opn\rk{rk}
 \opn\projdim{proj\,dim} \opn\injdim{inj\,dim} \opn\rank{rank}
 \opn\depth{depth} \opn\grade{grade} \opn\height{height}
 \opn\embdim{emb\,dim} \opn\codim{codim}
 
 \opn\Tr{Tr} \opn\bigrank{big\,rank}
 \opn\superheight{superheight}\opn\lcm{lcm}
 \opn\trdeg{tr\,deg}
 \opn\reg{reg} \opn\lreg{lreg} \opn\ini{in} \opn\lpd{lpd}
 \opn\size{size} \opn\sdepth{sdepth}
 \opn\link{link}\opn\fdepth{fdepth}\opn\lex{lex}
 \opn\tr{tr}
 \opn\type{type}
 \opn\gap{gap}
 \opn\arithdeg{arith-deg}
 \opn\revlex{revlex}
 \opn\div{div} \opn\Div{Div} \opn\cl{cl} \opn\Cl{Cl}
 \opn\Spec{Spec} \opn\Supp{Supp} \opn\supp{supp} \opn\Sing{Sing}
 \opn\Ass{Ass} \opn\Min{Min}\opn\Mon{Mon}
 \opn\Ann{Ann} \opn\Rad{Rad} \opn\Soc{Soc}
 \opn\Im{Im} \opn\Ker{Ker} \opn\Coker{Coker} \opn\Am{Am}
 \opn\Hom{Hom} \opn\Tor{Tor} \opn\Ext{Ext} \opn\End{End}
 \opn\Aut{Aut} \opn\id{id}
 
 \opn\nat{nat}
 \opn\pff{pf}
 \opn\Pf{Pf} \opn\GL{GL} \opn\SL{SL} \opn\mod{mod} \opn\ord{ord}
 \opn\Gin{Gin} \opn\Hilb{Hilb}\opn\sort{sort}
 \opn\PF{PF}\opn\Ap{Ap}
 \opn\mult{mult}
 \opn\bight{bight}
 \opn\aff{aff}
 \opn\relint{relint} \opn\st{st}
 \opn\lk{lk} \opn\cn{cn} \opn\core{core} \opn\vol{vol}  \opn\inp{inp} \opn\nilpot{nilpot}
 \opn\link{link} \opn\star{star}\opn\lex{lex}\opn\set{set}
 \opn\width{wd}
 \opn\Fr{F}
 \opn\QF{QF}
 \opn\G{G}
 \opn\type{type}\opn\res{res}
 \opn\conv{conv}
 \opn\Deg{Deg}
 \opn\Sym{Sym}
 \opn\gr{gr}
 
 \def\pot#1#2{#1[\kern-0.28ex[#2]\kern-0.28ex]}

 \opn\dirlim{\underrightarrow{\lim}}
 \opn\inivlim{\underleftarrow{\lim}}
 \let\union=\cup

 \def\Implies{\ifmmode\Longrightarrow \else
         \unskip${}\Longrightarrow{}$\ignorespaces\fi}
 \def\implies{\ifmmode\Rightarrow \else
         \unskip${}\Rightarrow{}$\ignorespaces\fi}
 \def\iff{\ifmmode\Longleftrightarrow \else
         \unskip${}\Longleftrightarrow{}$\ignorespaces\fi}

 \let\:=\colon
 \newtheorem{Theorem}{Theorem}[section]
 \newtheorem{Lemma}[Theorem]{Lemma}
 \newtheorem{Corollary}[Theorem]{Corollary}
 
 \newtheorem{Remark}[Theorem]{Remark}
 
 \newtheorem{Example}[Theorem]{Example}

 \let\epsilon\varepsilon
 \let\kappa=\varkappa
 \def\qed{\ifhmode\textqed\fi
       \ifmmode\ifinner\quad\qedsymbol\else\dispqed\fi\fi}
 \def\textqed{\unskip\nobreak\penalty50
        \hskip2em\hbox{}\nobreak\hfil\qedsymbol
        \parfillskip=0pt \finalhyphendemerits=0}
 \def\dispqed{\rlap{\qquad\qedsymbol}}
 
 \opn\dis{dis}
 \def\pnt{{\raise0.5mm\hbox{\large\bf.}}}
 
 \opn\Lex{Lex}

\begin{document}

\title{Finite distributive lattices, polyominoes and ideals of K\"onig type}
\author{J\"urgen Herzog and Takayuki Hibi
}

\date{\today}

\address{J\"urgen Herzog, Fachbereich Mathematik, Universit\"at Duisburg-Essen, Campus Essen, 45117 Essen, Germany}
\email{juergen.herzog@uni-essen.de}
\address{Takayuki Hibi, Department of Pure and Applied Mathematics, Graduate School of Information Science and Technology,
Osaka University, Toyonaka, Osaka 560-0043, Japan}
\email{hibi@math.sci.osaka-u.ac.jp}

\thanks{The second author was partially supported by JSPS KAKENHI 19H00637.  
}



\begin{abstract}
Finite distributive lattices whose join-meet ideals are of K\"onig type  will be classified.  Furthermore, a class of polyominoes whose polyomino ideals are of K\"onig type will be studied.
\end{abstract}

\maketitle

\setcounter{tocdepth}{1}

\section*{Introduction}

Inspired by K\"onig's theorem in the classical graph theory, an ideal of K\"onig type is introduced in \cite{HHM}.  In the present paper, binomial ideals of K\"onig type arising from finite distributive lattices (\cite{EH},\cite{hibi87}), and those from polyominoes (\cite{Q}) will be studied.

\section{Ideals of K\"onig type
}
Recall from \cite{HHM} what is an ideal of K\"onig type.  Let $S = K[x_1, \ldots, x_n]$ denote the polynomial ring in $n$ variables over a field $K$ with each $\deg x_i = 1$ and let $I \subset S$ be a graded ideal of height $h$.  We say that $I$ is {\em of K\"onig type} if there exist (i) a sequence $f_1, \ldots, f_h$ of homogeneous polynomials which forms part of a minimal system of homogeneous generators of $I$ and (ii) a monomial order $<$ on $S$ for which ${\rm in}_<(f_1), \ldots, {\rm in}_<(f_h)$ is a regular sequence.  More precisely, we say that $I$ is of K\"onig type with respect to the sequence $f_1, \ldots, f_h$ and the monomial order $<$.

\section{Finite distributive lattices}
Let $P$ be a finite partially ordered set (poset, for short) with $|P|=d$ and $L = \Jc(P)$ the finite distributive lattice (\cite[pp.~156--159]{HHgtm}) consisting of poset ideals of $P$, ordered by inclusion.  In other words, $P$ is the subposet of $L$ consisting of join-irreducible elements of $L$.  A subset $\{a_{i_0}, a_{i_1}, \ldots, a_{i_q}\}$ of $L$ of the form
$
a_{i_0} < a_{i_1} < \cdots < a_{i_q}
$
is called a {\em chain} of $L$ of {\em length} $q$.  It follows that the length of every maximal chain of $L$ is equal to $d$.  The {\em rank} of $a \in L$ is the maximal length of chains of $L$ of the form
$
a_{i_0} < a_{i_1} < \cdots < a_{i_r} = a.
$
Let $\rank_L(a)$ denote the rank of $a \in L$.  Let $\rho_L(i)$ denote the number of $a \in L$ with $\rank_L(a) = i$, where $0 \leq i \leq d$.  Especially $\rho_L(0) = \rho_L(d) = 1$.  We say that $\xi \in L$ is an {\em apex} of $L$ if $\rank_L(\xi) = i$ and $\rho_L(i) = 1$.  In other words, $\xi \in L$ is an apex of $L$ if, for each $a \in L$ with $a \neq \xi$, one has either  $\xi < a$ or $\xi > a$.  In particular, the unique minimal element $0_L$ of $L$ and the unique maximal element $1_L$ of $L$ are apexes of $L$.  A finite distributive lattice $L \neq \{0_L, 1_L\}$ is called {\em simple} if there is no apex of $L$ except for $0_L$ and $1_L$.

Let $S = K[\{x_a\}_{a \in L}]$ denote the polynomial ring in $|L|$ variables over a field $K$ and, for $a\in L$ and $b \in L$ which are incomparable in $L$, one introduces the binomial
\[
f_{a,b} = x_a x_b - x_{a\wedge b} x_{a \vee b}.
\]
The binomial ideal $I_L$ which is generated by those $f_{a,b}$ for which $a\in L$ and $b \in L$ are incomparable in $L$ is introduced in \cite{hibi87} and is called the {\em join-meet ideal} of $L$ (\cite{EH}).
Let $\Gc_L$ be the set of binomials $f_{a,b}$, where $a\in L$ and $b \in L$ are incomparable in $L$.  It is shown \cite{hibi87} that $\Gc_L$ is a minimal system of homogeneous generators of $I_L$.  Furthermore, if $<_{\rm rev}$ is a reverse lexicographic order on $S$ for which $a <_{\rm rev} b$ if $\rank(a) < \rank(b)$, then $\Gc_L$ is the reduced Gr\"obner basis of $I_L$ with respect to $<_{\rm rev}$.

Since $S/I_L$ is Cohen--Macaulay with $\dim S/I_L = d + 1$, one has
\[
\height I_L = |L| - (d + 1).
\]
It then follows that if $I_L$ is of K\"onig type, then $2(|L| - (d + 1)) \leq |L|$, equivalently, $|L| \leq 2(d + 1)$. By an abuse of language, we say that a finite distributive lattice $L$ is of K\"onig type if its join-meet ideal $I_L$ is of K\"onig type with respect to a sequence $f_{a_1, b_1}, f_{a_2,b_2}, \ldots, f_{a_h, b_h}$, where $h = |L| - (d + 1)$, and a monomial order $<$ on $S$.

\begin{Lemma}
\label{Mozart}
Let $L' = \Jc(P')$ and $L'' = \Jc(P'')$ be finite distributive lattice.  Let $L = L' \bigoplus L''$ be the ordinal sum \cite[p.~246]{EC} of $L'$ and $L''$.  Then the finite distributive lattice $L$ is of K\"onig type if and only if each of $L'$ and $L''$ is of K\"onig type.
\end{Lemma}

\begin{proof}
One has  $L = \Jc(P)$, where $P = P' \bigoplus \{a\} \bigoplus P''$.  Let $|P'| = d'$ and $|P''| = d''$.  Since  $\height I_{L'} = |L'| - (d' + 1)$ and $\height I_{L''} = |L''| - (d'' + 1)$, it follows that
\[
\height I_L = |L| - (d' + d'' + 2) = \height I_{L'} + \height I_{L''}.
\]
Each binomial $f_{a, b} \in I_L$ belong to either $I_{L'}$ or $I_{L''}$.  It then follows that $L$ is of K\"onig type if and only if each of $L'$ and $L''$ is of K\"onig type, as required.
\end{proof}

A finite distributive lattice $L$ is {\em decomposable},  if $L$ is of the form $L = L' \bigoplus L''$, where each of $L'$ and $L''$ is a distributive lattice.  It follows that $L$ is non-decomposable if and only if, for apexes $\xi$ and $\xi'$ of $L$ with $\xi < \xi'$, one has $\rank(\xi') - \rank (\xi) \geq 2$.  Every finite simple distributive lattice is non-decomposable.

In general, we say that a finite simple distributive lattice $L$ is {\em quasi-thin} if $\rho_L(i) \leq 3$ for each $1 \leq i < d$.  When $L$ is quasi-thin, one introduces $\theta(L) = |\{ i : \rho_L(i) = 3\}|$.  A finite quasi-thin distributive lattice $L$ is called {\em thin} if $\theta(L) = 0$.

\begin{Lemma}
\label{Chopin}
Suppose that a finite simple distributive lattice $L$ is of K\"onig type.  Then $L$ is quasi-thin with $\theta(L) \leq 2$.
\end{Lemma}

\begin{proof}
Let $L = \Jc(P)$ with $|P| = d$.  Since $L$ is simple and since $|L| \leq 2(d + 1)$, it follows that
\[
2\geq |L| - 2d = \sum_{i=1}^{d-1} (\rho_L(i) - 2).
\]
Suppose that there is $1 \leq i_0 < d$ with $\rho_L(i_0)\geq  4$.  Then $\rho_L({i}) = 2$ for $1 \leq i < d$ with $i \neq i_0$.  If $L$ is planar \cite[p.~436]{AHH}, then $P$ possesses the subposet $\{a_1, \ldots, a_6\}$ with the partial order $a_1 < a_2 < a_3$ and $a_4 < a_5 < a_6$.  Hence $\rho_L({i_0-1}) \geq 3$ and $\rho_L({i_0+1}) \geq 3$, a contradiction.  If $L$ is non-planar, then the boolean lattice \cite[Example 9.1.6 (a)]{HHgtm} on $[n] = \{1,\ldots,n\}$ with $n \geq 4$ is an interval of $L$.  Thus there is $1 \leq i'_0 < d$ with $i'_0 \neq i_0$ and $\rho_L({i'_0}) \geq 4$, a contradiction.  It follows that $L$ is quasi-thin.  Furthermore, since $\sum_{i=1}^{d-1} (\rho_L(i) - 2) \leq 2$, one has $\theta(L) \leq 2$.
\end{proof}

\section{Finite simple distributive lattices of K\"onig type}
Now, in the present section, a classification of finite simple distributive lattices of K\"onig type will be done.  Lemma \ref{Chopin} says that a finite simple distributive lattice $L$ of K\"onig type is quasi-thin with $\theta(L) \leq 2$.  Hence the first step our job is to classify finite simple quasi-thin distributive lattices $L$ with $\theta(L) \leq 2$.

First, suppose that $L$ is non-planar.  Then the boolean lattice on $\{1,2,3\}$ is an interval of $L$ and $\theta(L) = 2$.  Each of the finite distributive lattices of Figure $1$ is simple, non-planar and quasi-thin.
\begin{figure}[h]
\begin{center}
\includegraphics[scale=0.9]{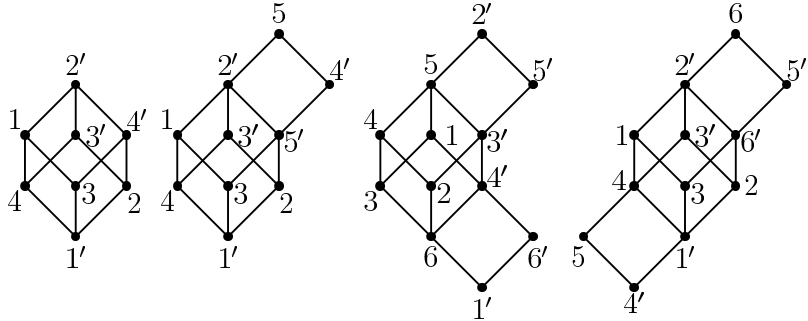}
\end{center}
\caption{Simple non-planar quasi-thin with $\theta(L)\leq 2$}
\label{horizontal}
\end{figure}

Second, suppose that $L$ is planar.  We divide finite simple planar quasi-thin distributive lattices $L$ with $\theta(L) \leq 2$ into $5$ classes as follows:
\begin{itemize}
\item
(type $0$) thin;
\item
(type $1$) quasi-thin with $\theta(L) = 1$;
\item
(type $2_a$) quasi-thin with $\theta(L) = 2$ and $\rho_L(i_0) = \rho_L(i_0 + 1) = 3$;
\item
(type $2_b$) quasi-thin with $\theta(L) = 2$ and $\rho_L(i_0) = \rho_L(i_0 + 2) = 3$;
\item
(type $2_c$) quasi-thin with $\theta(L) = 2$ and $\rho_L(i_0) = \rho_L(j_0) = 3$ with $j_0 \geq i_0 + 3$.
\end{itemize}
Let $L, L'$ and $L''$ denote the distributive lattices in Figure $2$ from the right to the left.  The interval $[h, a']$ of $L''$ is of type $0$.  Each of the intervals $[1', a']$ and $[h, 5']$ of $L''$ is of type $1$.  The distributive lattice $L$ is of type $2_a$ with $i_0 = 2$.  The distributive lattice $L'$ is of type $2_b$ with $i_0 = 2$.  The distributive lattice $L''$ is of type $2_c$ with $i_0 = 2$ and $j_0 = 13$.


\begin{figure}[h]
\begin{center}
\includegraphics[scale=0.9]{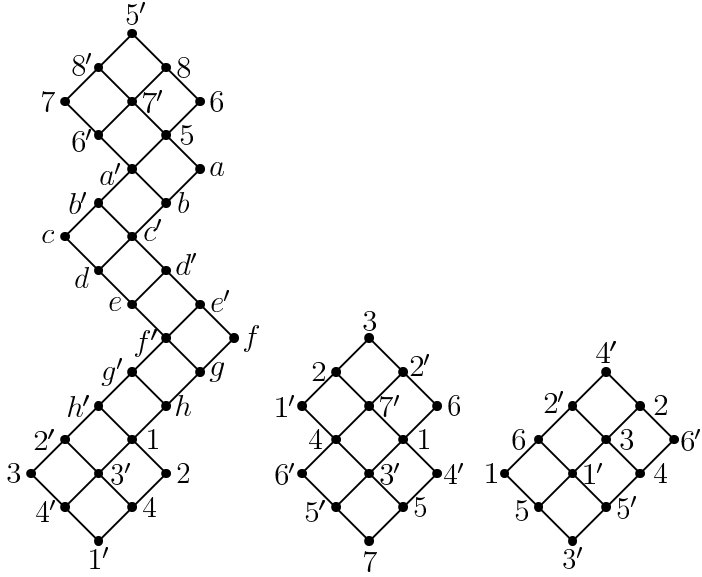}
\end{center}
\caption{Simple planar quasi-thin with $\theta(L)\leq 2$}
\label{long}
\end{figure}


\begin{Example}
\label{Beethoven}
{\em
Each of the distributive lattices of Figure $1$ is of K\"onig type with respect to the sequence
\begin{eqnarray}
\label{Boston}
x_{1}x_{1'} - x_{1 \wedge 1'}x_{1 \vee 1'}, x_{2}x_{2'} - x_{2 \wedge 2'}x_{2 \vee 2'}, \ldots
\end{eqnarray} and the lexicographic order induced by the ordering
\begin{eqnarray}
\label{Sydney}
x_1 > x_2 > \cdots > x_{1'} > x_{2'} > \cdots.
\end{eqnarray}
The distributive lattice $L''$ of Figure $2$ is of K\"onig type with respect to the sequence
\[
x_{1}x_{1'} - x_{1 \wedge 1'}x_{1 \vee 1'}, \ldots,
x_{a}x_{a'} - x_{a \wedge a'}x_{a \vee a'}, \ldots, x_{h}x_{h'} - x_{h \wedge h'}x_{h \vee h'}
\]
and the lexicographic order induced by the ordering
\[
x_a > \cdots > x_{h} > x_{1} > \cdots > x_{8} > x_{a'} > \cdots > x_{h'} > x_{1'} > \cdots > x_{8'}.
\]
Each of the distributive lattices $L$ and $L'$ of Figure $2$ is of K\"onig type with respect the sequence (\ref{Boston}) and the lexicographic order induced by the ordering (\ref{Sydney}).
}
\end{Example}

\begin{Lemma}
\label{Nightingale}
Let $L = \Jc(P)$ with $|P| = d$ be a finite distributive lattice of K\"onig type with respect to  $f_{a_1, b_1}, f_{a_2,b_2}, \ldots, f_{a_h, b_h}$ and a lexicographic order $<$ on $S$. Here $h = |L| - (d + 1)$.  Suppose that ${\rm in}_{<}(f_{a_1, b_1}) = x_{a_1}x_{b_1}, \, \rank_L(a_1) = \rank_L(b_1) = d - 1$ and that $a_1 \vee b_1$ is the unique maximal element $1_L$ of $L$.  Let $L'$ denote the finite distributive lattice which is obtained by adding new elements $c$ and $e$ to $L$, where
\[
a_1 \vee b_1 < e, \, \, \, b_1 < c < e, \, \, \, a_1 \vee b_1 \neq c.
\]
Then $L'$ is of K\"onig type with respect to  $f_{a_2, b_2}, f_{a_3,b_3}, \ldots, f_{a_h, b_h}, f', f''$, where
\[
f' = x_{a_1}x_{c} - x_{a_1 \wedge b_1}x_{e}, \, \, \, f'' = x_{b_1}x_{e} - x_{a_1 \vee b_1}x_{c}
\]
and a lexicographic order $<'$ on $S' = S[x_c, x_e]$ with
\[
{\rm in}_{<'}(f') = x_{a_1}x_{c}, \, \, \, {\rm in}_{<'}(f'') = x_{b_1}x_{e}.
\]
\end{Lemma}


\begin{figure}[h]
\begin{center}
\includegraphics[scale=0.9]{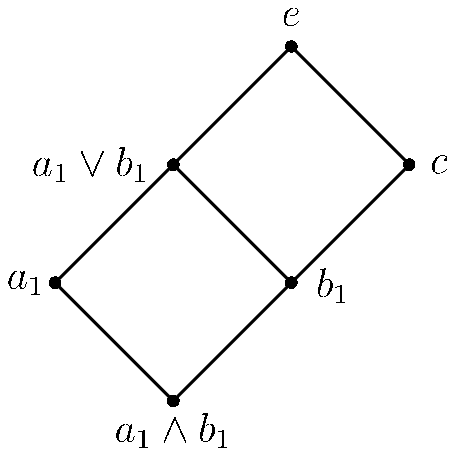}
\end{center}
\caption{$L'=L\union \{c,e\}$}
\label{long}
\end{figure}


\begin{proof}
If $x_{a_1}$ is bigger than each of $x_{b_1}, x_{a_1 \wedge b_1}, x_{a_1 \vee b_1}$ with respect to $<$ and if $<$ is the lexicographic order on $S$ induced by the ordering
$
\cdots > x_{a'} > x_{a_1} > x_{a''} > \cdots,
$
then the lexicographic order $<'$ on $S'$ induced by the ordering
\[
\cdots > x_{a'} > x_{a_1} > x_{e} > x_{c} > x_{a''} > \cdots
\]
satisfies ${\rm in}_{<'}(f') = x_{a_1}x_{c}$ and ${\rm in}_{<'}(f'') = x_{b_1}x_{e}$.  On the other hand, if $x_{b_1}$ be bigger than each of $x_{a_1}, x_{a_1 \wedge b_1}, x_{a_1 \vee b_1}$ and if $<$ is the lexicographic order on $S$ induced by the ordering
$
\cdots > x_{b'} > x_{b_1} > x_{b''} > \cdots,
$
then the lexicographic order $<'$ on $S'$ induced by the ordering
\[
\cdots > x_{b'} > x_{b_1} > x_{c} > x_{e} > x_{b''} > \cdots
\]
satisfies ${\rm in}_{<'}(f') = x_{a_1}x_{c}$ and ${\rm in}_{<'}(f'') = x_{b_1}x_{e}$.  Hence $L'$ is of K\"onig type with respect to $f_{a_2, b_2}, \ldots, f_{a_h, b_h}, f', f''$ and $<'$ on $S'$, as required.
\end{proof}

\begin{Lemma}
\label{Nightingale*}
Let $L = \Jc(P)$ with $|P| = d$ be a finite distributive lattice of K\"onig type with respect to  $f_{a_1, b_1}, f_{a_2,b_2}, \ldots, f_{a_h, b_h}$ and a lexicographic order $<$ on $S$. Here $h = |L| - (d + 1)$.   Suppose that ${\rm in}_{<}(f_{a_1, b_1}) = x_{a_1 \wedge b_1}x_{a_1 \vee b_1}, \, \rank_L(a_1) = \rank_L(b_1) = d - 1$ and that $a_1 \vee b_1$ is the  unique maximal element $1_L$ of $L$.  Let $L'$ denote the finite distributive lattice which is obtained by adding new elements $c$ and $e$ to $L$, where
\[
a_1 \vee b_1 < e, \, \, \, b_1 < c < e, \, \, \, a_1 \vee b_1 \neq c.
\]
Then $L'$ is of K\"onig type with respect to $f_{a_2, b_2}, f_{a_3,b_3}, \ldots, f_{a_h, b_h}, f', f''$, where
\[
f' = x_{a_1 \wedge b_1}x_{e} - x_{a_1}x_{c}, \, \, \, f'' = x_{a_1 \vee b_1}x_{c} - x_{b_1}x_{e}
\]
and a lexicographic order $<'$ on $S' = S[x_c, x_e]$ with
\[
{\rm in}_{<'}(f') = x_{a_1 \wedge b_1}x_{e}, \, \, \, {\rm in}_{<'}(f'') = x_{a_1 \vee b_1}x_{c}.
\]
\end{Lemma}

\begin{proof}
In the proof of Lemma \ref{Nightingale}, replacing the orderings with
\[
\cdots > x_{a'} > x_{a_1 \vee b_1} > x_{e} > x_{c} > x_{a''} > \cdots
\]
and
\[
\cdots > x_{b'} > x_{a_1 \wedge b_1} > x_{c} > x_{e} > x_{b''} > \cdots
\]
yields the desired result.
\end{proof}

\begin{Remark}
{\em
A dual version  of Lemmata \ref{Nightingale} and \ref{Nightingale*}, which can be easily obtained by replacing $1_L$ with $0_L$, is also valid with its dual proof.
}
\end{Remark}

\begin{Lemma}
\label{Chaikovskii}
A finite simple distributive lattice $L$ is of K\"onig type if and only if  $L$ is quasi-thin with $\theta(L) \leq 2$.
\end{Lemma}

\begin{proof}
The ``only if'' part is known (Lemma \ref{Chopin}).  On the other hand, the observation made in Example \ref{Beethoven} together with repeated application of the technique introduced in  Lemma \ref{Nightingale} and Lemma \ref{Nightingale*} guarantees that every finite quasi-thin distributive lattice with $\theta(L) \leq 2$ is of K\"onig type, as desired.
\end{proof}

\section{Classification of distributive lattices of K\"onig type}
Let $L = \Jc(P)$ with $|P| \geq 2$ be a finite non-decomposable distributive lattice and
\[
0_L = \xi_0 < \xi_1 < \cdots < \xi_s = 1_L
\]
the apexes of $L$.  Let $L_i = [\xi_{i-1}, \xi_i] = \Jc(P_i)$ with $|P_i| = d_i \geq 2$, where $1 \leq i \leq s$.  It then follows that $P = P_1 \bigoplus \cdots \bigoplus P_s$ and $d = d_1 + \cdots + d_s$.

\begin{Theorem}
\label{Johann_Sebastian_Bach}
Let $L$ be a finite non-decomposable distributive lattice and
\[
0_L = \xi_0 < \xi_1 < \cdots < \xi_s = 1_L
\]
the apexes of $L$.  Let $L_i = [\xi_{i-1}, \xi_i]$, where $1 \leq i \leq s$.  Then $L$ is of K\"onig type if and only if the following conditions are satisfied:
\begin{itemize}
\item
Each $L_i$ is quasi-thin with $\theta(L_i) \leq 2$.
\item
If $\theta(L_i) = 2$ and $\theta(L_{i'}) = 2$ with $1 \leq i < i' \leq s$, then there is $i < j < i'$ with $\theta(L_{j}) = 0$.
\end{itemize}
\end{Theorem}

\begin{proof}
{\bf (``only if'')}
Suppose that a finite non-decomposable distributive lattice $L$ is of K\"onig type with respect to a sequence $f_{a_1, b_1}, \ldots, f_{a_h, b_h}$ and a monomial order $<$ on $S = K[\{x_a\}_{a \in L}]$, where $L= \Jc(P)$ with $|P| = d \geq 2$ and $h = |L| - (d + 1)$.  Let $L_i = \Jc(P_i)$ with $|P_i| = d_i \geq 2$.  
Let $u_j = {\rm in}_<(f_{a_{j},b_{j}})$.  Let $u_{j_1}, \ldots, u_{j_{g_i}}$ belong to $S_i = K[\{x_a\}_{a \in L_i}]$.  It then follows that
\begin{eqnarray*}
g_i \leq \grade({\rm in}_<(I_{L_i})) & = & \height({\rm in}_<(I_{L_i})) \\ & \leq & |L_i| - \dim(S/{\rm in}_<(I_{L_i})) \\
& = & |L_i| - \dim(S/I_{L_i}) \\
& = & |L_i| - (d_i + 1).
\end{eqnarray*}
Hence
\begin{eqnarray*}
h = \sum_{i=1}^{s} g_i & \leq & \sum_{i=1}^{s}(|L_i| - (d_i + 1)) \\
& = & |L| + (s - 1) - d - s \\
& = & |L| - (d + 1).
\end{eqnarray*}
Since $h = |L| - (d + 1)$, one has $g_i = |L_i| - (d_i + 1)$ for each $1 \leq i \leq s$.  Thus each $L_i$ is of K\"onig type.  In other words, each $L_i$ is quasi-thin with $\theta(L_i) \leq 2$.

Now, suppose that $\theta(L_{i_0}) = 2$ and $\theta(L_{i_0'}) = 2$ with $1 \leq i_0 < i_0' \leq s$ and that $\theta(L_{j}) = 1$ for each $i_0 < j < i_0'$.
Since $g_i = |L_i| - (d_i + 1)$, it follows that
\begin{eqnarray*}
\sum_{i=i_0}^{i_0'} 2 g_i & = &  \sum_{i=i_0}^{i_0'}|L_i| - (i'_0 - i_0 - 1) \\
& = & |L'| + (i_0' - i_0) - (i'_0 - i_0 - 1) \\
& = & |L'| + 1.
\end{eqnarray*}
On the other hand, one has $\sum_{i=i_0}^{i_0'} 2 g_i \leq |L'|$, a contradiction.

\medskip

\noindent
{\bf (``if'')}  Lemma \ref{Chaikovskii} says that each of the simple distributive lattices $L_i$ is of K\"onig type with respect to a sequence $f_{a^{(i)}_1, b^{(i)}_1}, \ldots, f_{a^{(i)}_{h_i}, b^{(i)}_{h_i}}$, where $h_i = |L_i| - (d_i + 1)$, and a monomial order $<^{(i)}$ on $S_i = K[\{x_a\}_{a \in L_i}]$.  Let ${\rm in}_{<^{(i)}}(f_{a^{(i)}_j, b^{(i)}_j}) = x_{c^{(i)}_j} x_{e^{(i)}_j}$ for each $1 \leq i \leq s$ and $1 \leq j \leq h_i$.  Furthermore, set $\Ac_i = \{c^{(i)}_j, e^{(i)}_j : 1 \leq j \leq h_i\}$ for each $1 \leq i \leq s$.  The crucial facts observed in Example \ref{Beethoven} are as follows:
\begin{itemize}
\item[(i)]
If $\theta(L_i) = 0$, then one can make neither $\xi_{i-1}$ nor $\xi_{i}$ belongs to $\Ac_i$;
\item[(ii)]
Let $\theta(L_i) = 1$.  Then there exists a lexicographic order $<_*$ on $S_i$ for which $\xi_{i-1} \in \Ac_i$ and $\xi_{i} \not\in \Ac_i$.  Furthermore, there exists a lexicographic order $<^*$ on $S_i$ for which $\xi_{i-1} \not\in \Ac_i$ and $\xi_{i} \in \Ac_i$.
\end{itemize}
Now, it follows from these facts (i) and (ii) together with $h = \sum_{i=1}^{s} h_i$ that ``if'' part turns out to be true.
\end{proof}

\begin{figure}[h]
\begin{center}
\includegraphics[scale=0.9]{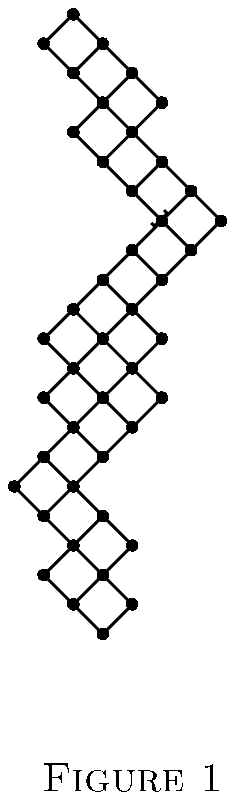}
\end{center}
\caption{Repeated application of Lemmata~\ref{Nightingale} and \ref{Nightingale*}}
\label{long}
\end{figure}

\section{Polyominoes and polyomino ideals}
Recall from \cite{Q} fundamental materials on polyominoes and their binomial ideals.  One regards $\NN^2$ as a infinite poset with the natural partial order defined by setting $(i, j) \leq (i', j')$ if $i \leq i'$ and $j \leq j'$.  Let $\ab, \bb \in \NN^2$ with $\ab \leq \bb$.  Then the set $[\ab, \bb] = \{ \cb \in \NN^2 : \ab \leq \cb \leq \bb \}$ is an interval of $\NN^2$.  If $\ab =(i,j), \bb = (i',j')$ with $i < i'$ and $j < j'$, then the interval $[\ab, \bb]$ is called {\em proper}.  The {\em corners} of the proper interval $[\ab, \bb]$ are $\ab, \bb$ and $\cb = (i', j), \db = (i, j')$.  We say that $\ab$ and $\bb$ are the diagonal corners and that $\cb$ and $\db$ are anti-diagonal corners of $[\ab, \bb]$.
The interval $C = [\ab, \bb]$ with $\bb = \ab + (1,1)$ is called a {\em cell} of $\NN^2$.  Let $\cb, \db$ be anti-diagonal corners of the cell $C = [\ab, \bb]$.  The set of {\em vertices} of $C$ is $V(C) = \{\ab, \bb, \cb, \db\}$ and the set of {\em edges} of $C$ is $E(C) = \{ \{\ab, \cb \}, \{\ab, \db\}, \{\bb, \cb\}, \{\bb, \db\} \}$.
Let $[\ab, \bb]$ be a proper interval of $\NN^2$.  A cell $C = [\ab', \bb']$ of $\NN^2$ is called a cell of $[\ab, \bb]$ if $\ab \leq \ab'$ and $\bb' \leq \bb$.

Let $\Pc$ be a finite collection of cells of $\NN^2$.  The vertex set of $\Pc$ is $V(\Pc) = \cup_{C \in \Pc} V(C)$ and the edge set of $\Pc$ is $E(C) = \cup_{C \in \Pc} E(C)$.  A vertex $\ab \in V(\Pc)$ is called an {\em interior vertex} of $\Pc$ if $\ab$ is a vertex of four distinct cells of $\Pc$, otherwise it is called a {\em boundary vertex} of $\Pc$.  Let $C$ and $D$ be two cells of $\Pc$.  Then $C$ and $D$ are {\em connected} in $\Pc$ if there is a sequence of cells of $\Pc$ of the form $C = C_1, \ldots, C_m = D$ for which $C_i \cap C_{i+1}$ is an edge of $C_i$ for $i = 1, \ldots, m - 1$.  A {\em polyomino} is a finite collection $\Pc$ of cells of $\NN^2$ for which any two cells of $\Pc$ is connected in $\Pc$.

Let $\Pc$ be a polyomino and $S = K[\{x_\ab\}_{\ab \in V(\Pc)}]$ the polynomial ring in $|V(\Pc)|$ variables over a field $K$.  A proper interval $[\ab, \bb]$ of $\NN^2$ is called an {\em inner interval} of $\Pc$ if each cell of $[\ab, \bb]$ belongs to $\Pc$.  Now, for each inner interval $[\ab, \bb]$ of $\Pc$, one introduces the binomial $f_{\ab,\bb} = x_{\ab}x_{\bb}-x_{\cb}x_{\db}$, where $\cb$ and $\db$ are the anti-diagonal corners of $[\ab, \bb]$  The binomial $f_{\ab,\bb}$ is called an {\em inner $2$-minor} of $\Pc$.  The {\em polyomino ideal} of $\Pc$ is the binomial ideal $I_\Pc$ which is generated by the inner $2$-minors of $\Pc$.  Furthermore, we write $K[\Pc]$ for the quotient ring $S/I_\Pc$.  We say that $\Pc$ is of K\"onig type if $I_\Pc$ is of K\"onig type with respect to a sequence $f_{\ab_1,\bb_1}, \ldots, f_{\ab_h,\bb_h}$ of inner $2$-minors, where $h = \height I_\Pc$, and a monomial order $<$ on $S$.

In order to study of polyominoes of K\"onig type, to find a combinatorial formula to compute $\height I_\Pc$ is indispensable.

\begin{Theorem}
\label{hotel_herzog}
Let $\Pc$ be a polyomino for which each $x_{ij}$ with $(i, j) \in V(\Pc)$ is a non-zero divisor modulo $I_\Pc$.  Then $\height I_\Pc \leq |\Pc|$.  In particular, when $I_\Pc$ is prime, one has $\height I_\Pc \leq |\Pc|$.
\end{Theorem}

\begin{proof}
Let $S = K[x_{ij} : (i, j) \in V(\Pc)]$ and $\xb = \prod_{ (i, j) \in V(\Pc)} x_{ij} \in S$.  One claims that $I_\Pc S_\xb$ is generated by the $2$-minors corresponding to the cells of $\Pc$.  The choice of $\xb$ guarantees that $\height I_\Pc = \height I_\Pc S_\xb$.  The claim says that $I_\Pc S_\xb$ is generated by $|\Pc|$ elements, from which $\height I_\Pc S_\xb \leq |\Pc|$ follows.  Thus $\height I_\Pc \leq |\Pc|$, as desired.

It remains to proof the claim.  
Let $i_1 < i_2 < i_3$ and $j_1 < j_2$, and assume that $[(i_1, j_1), (i_3, j_2)]$ is an inner interval of $\Pc$. Then each of $[(i_1, j_1), (i_2, j_2)]$ and $[(i_2, j_1), (i_3, j_2)]$ is also an inner interval of $\Pc$.  Let $\ab = (i_1,j_1), \cb = (i_2,j_1), \db = (i_2,j_2)$ and $\bb = (i_3, j_2)$.  Then
\[
x_{i_1j_2} f_{\cb,\bb} - x_{i_2j_2} f_{\ab,\bb} + x_{i_3j_2} f_{\ab,\db} = 0.
\]
This equation shows that in $S_{\xb}$ the 2-minor $f_{\ab,\bb}$ of the big interval $[\ab, \bb]$ is a linear combination of the 2-minors $f_{\ab,\db}$ and $f_{\cb,\bb}$ of the smaller intervals $[\ab, \db]$ and $[\cb, \bb]$.


\begin{figure}[h]
\begin{center}
\includegraphics[scale=0.9]{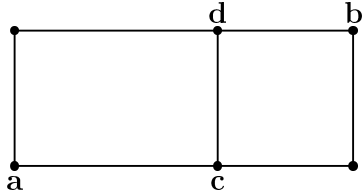}
\end{center}
\caption{Vertical splitting of an interval}
\label{Vertical_800x800_p1.eps}
\end{figure}


Now one proves the claim by showing that, for each inner interval $[\ab,\bb]$ of $\Pc$, the $2$-minor $f_{\ab,\bb}$ is in $S_\xb$ a linear combination of the $2$-minors corresponding to cells of $[\ab,\bb]$.  One proceeds by induction on the number of cell columns.  If there is only one column, then one proceeds by the length of this column. If this column consists of only one cell, then the job is done.  If the column consists of more than one column, then the column can be splitted into two shorter columns as indicated in Figure~\ref{horizontal}.

\begin{figure}[h]
\begin{center}
\includegraphics[scale=0.9]{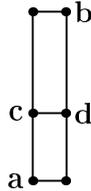}
\end{center}
\caption{Horizontal splitting of an interval}
\label{horizontal}
\end{figure}

Then similarly as in the vertical case, in $S_\xb$ the 2-minor of the whole column is a linear combination of the $2$-minors of the two shorter columns.  The induction hypothesis guarantees that in $S_\xb$ each of the two shorter columns is a linear combination of the $2$-minors of their inner cells.  It then follows that in $S_\xb$ the $2$-minor of the whole column is a linear combination of the $2$-minors of its cells.

Now suppose that $[\ab,\bb]$ consists of more than one column. Then one splits $[\ab,\bb]$ as shown in Figure $5$ and applies the similar induction argument as was done for the first column of $[\ab,\bb]$, in order to deduce that, in $S_\xb$, the inner $2$-minor $f_{\ab,\bb}$ is a linear combination of the 2-minors corresponding to the cells of $[\ab,\bb]$.
\end{proof}

Recall from \cite{QSS} the following definitions and facts: let $\Pc$ be a polyomino.  An interval $[\ab, \bb]$ of $\Pc$ with $\ab = (i, j)$ and $\bb = (k, \ell)$ is called a {\em horizontal edge interval} of $\Pc$ if $j = \ell$ and if each $\{(r,j),(r+1,j)\}$ for $i \leq r < k$ is an edge of a cell of $\Pc$. If a horizontal edge interval of $\Pc$ is not strictly contained in any other horizontal edge interval of $\Pc$, then it is called maximal. Similarly one defines {\em vertical edge intervals} and maximal vertical edge intervals of $\Pc$.  Let $h(\Pc)$ denote the number of maximal horizontal edge intervals of $\Pc$ and $v(\Pc)$ the number of maximal vertical edge intervals of $\Pc$.  In \cite[Theorem 2.2]{QSS} it is shown that if $\Pc$ is a simple polyomino (\cite[p.~282]{Q}) , then $K[\Pc]$ is isomorphic to the edge ring $K[G(\Pc)]$, where $G(\Pc)$ is the bipartite graph with vertex decomposition $V(G(\Pc)) = V_1 \cup V_2$ with $|V_1| = h(\Pc)$ and $|V_2| = v(\Pc)$, and $|E(G(P))| = |V(\Pc)|$.  This implies in particular, that $\Pc$ is a prime polyomino, i.e., $K[\Pc]$ is an integral domain.

\begin{Corollary}
\label{Regensburg}
Let $P$ be a simple polyomino. Then $$\height I_\Pc = |V(\Pc)| - (h(\Pc)+v(\Pc) - 1).$$
\end{Corollary}

\begin{proof}

It is known \cite{OH} that $\dim K[G] = n - 1$ for a bipartite graph with $n$ vertices.  Applied to our case it follows that
\begin{eqnarray*}
\height I_\Pc & = & \embdim K[\Pc] - \dim K[\Pc] \\
& = & \embdim K[G(\Pc)] - \dim K[G(\Pc)] \\
& = & |V(\Pc)| - (h(\Pc) + v(\Pc) - 1),
\end{eqnarray*}
as desired.
\end{proof}

Together with Theorem \ref{hotel_herzog} one can now obtain

\begin{Corollary}
\label{Frankfurt}
Let $\Pc$ be a simple polyomino and suppose that there exist $|\Pc|$ inner $2$-minors, whose initial monomials form a regular sequence. Then $\Pc$ is of K\"onig type and
\[
|\Pc| = |V(\Pc)| - (h(\Pc) + v(\Pc) - 1).
\]
\end{Corollary}

Let $\Pc$ is a polyomino and $C = [\ab, \bb]$ a cell of $\Pc$ for which one of the vertices of $C$ is a boundary vertex of $\Pc$.  Let $\ab = (i, j)$ and $\bb = (i+1, j+1)$.  Let $\cb = (i, j+1)$ and $\db = (i+1, j)$ be the anti-diagonal corners of $[\ab, \bb]$.  Let, say, $\db$ and $\bb$ be boundary vertices of $\Pc$.  Let $\Pc'$ denote the new polyomino which is obtained by adding the new cell $[\db, \eb]$ to $\Pc$, where $\eb = (i+2, j+1)$.  It then follows that $\height I_{\Pc'} = \height I_\Pc + 1$.  Now, suppose that 
there is a sequence $f_{\ab_1,\bb_1}, \ldots, f_{\ab_{h},\bb_{h}} = f_{\ab,\bb}$ together with a lexicographic order $<$ on $S$ for which ${\rm in}_<(f_{\ab_1,\bb_1}), \ldots, {\rm in}_<(f_{\ab_h,\bb_h})$ form a regular sequence.


\begin{figure}[h]
\begin{center}
\includegraphics[scale=0.9]{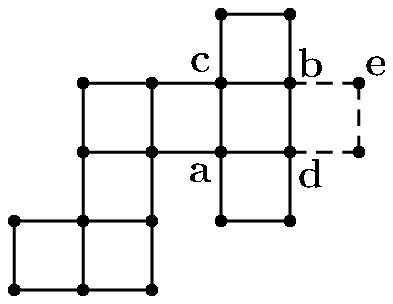}
\end{center}
\caption{Polyominoes $\Pc$ and $\Pc'$}
\label{prime}
\end{figure}


\begin{Lemma}
\label{Ittenbach}
Following the above situation and suppose that ${\rm in}_<(f_{\ab,\bb}) = x_\ab x_\bb$.  Then there is a lexicographic order $<'$ on $S[x_{\db'}, x_\eb]$, where $\db' = (i+2, j)$, for which
\[
{\rm in}_{<'}(f_{\ab_1,\bb_1}), \ldots, {\rm in}_{<'}(f_{\ab_{h-1},\bb_{h-1}}), x_\ab x_\eb = {\rm in}_{<'}(f_{\ab,\eb}), x_\bb x_{\db'} = {\rm in}_{<'}(f_{\db, \eb})
\]
form a regular sequence.
\end{Lemma}

Imitating the proofs   of the  Lemmata \ref{Nightingale} and \ref{Nightingale*}, one can prove Lemma \ref{Ittenbach} easily.  Furthermore, slight modifications of Lemma \ref{Ittenbach}, for example, when $\ab$ and $\db$ 
are boundary vertices of $\Pc$ and ${\rm in}_<(f_{\ab,\bb}) = x_\cb x_\db$, can be valid.

A vertex $\ab$ of a polyomino $\Pc$ is called {\em free} if $\ab$ belongs to exactly one cell of $\Pc$.  A cell $C$ of a polyomino $\Pc$ is called a {\em leaf} of $\Pc$ if two of the vertices of $C$ are free.  A simple polyomino $\Pc$ is called a {\em tree} if $\Pc$ possesses a leaf and if no inner interval of $\Pc$ is of the form $[\ab, \ab + (2,2)]$.  It follows from repeated application of Lemma \ref{Ittenbach} that

\begin{Corollary}
\label{Ittenbach_tree}
Every tree $\Pc$ is of K\"onig type with $\height I_\Pc = |\Pc|$.
\end{Corollary}


\begin{figure}[h]
\begin{center}
\includegraphics[scale=0.9]{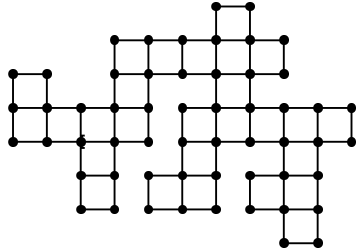}
\end{center}
\caption{A tree polyomino}
\label{tree}
\end{figure}


\begin{Example}
\label{cycle}
{\em
The polyomino $\Pc$ of Figure $9$ is of K\"onig type with respect to the sequence $f_{1,1'}, \ldots, f_{16, 16'}$ and the lexicographic order $<$ on $S$ induced by the ordering
$$
1 > \cdots > 16 > 1' > \cdots > 16'.
$$

}
\end{Example}

\begin{figure}[h]
\begin{center}
\includegraphics[scale=0.9]{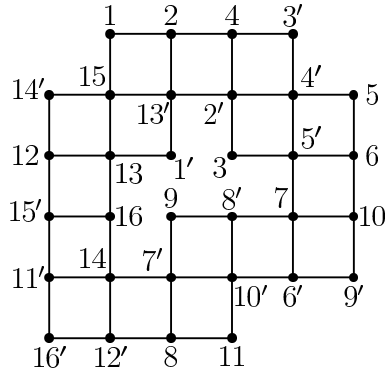}
\end{center}
\caption{A cycle polyomino}
\label{prime}
\end{figure}

It would be of interest to classify the polyominoes which are of K\"onig type.



{}

\end{document}